\documentclass[11pt]{article}
\usepackage[a4paper, margin=1in]{geometry} 
\usepackage{amsmath, amssymb, amsthm} 
\usepackage{graphicx} 
\usepackage{enumitem} 
\usepackage{mathpazo} 
\usepackage{titlesec} 
\usepackage{hyperref} 
\usepackage{nicefrac}

\usepackage{float} 
\usepackage{cite}
\titleformat{\section}{\large\bfseries}{\thesection.}{1em}{}
\titleformat{\subsection}{\normalsize\bfseries}{\thesubsection.}{1em}{}

\theoremstyle{definition}
\newtheorem{definition}{Definition}[section]

\theoremstyle{plain}
\newtheorem{theorem}{Theorem}[section]
\newtheorem{proposition}{Proposition}[section]

\everymath{\displaystyle}

\theoremstyle{remark}
\newtheorem{remark}{Remark}[section]

\title{\textbf{Formalising Propositional Information via Implication Hypergraphs}}
\author{
    \textsc{Vibhu Dalal} \\[0.6em]
    {\small \text{Université Paris Cité}}
}
\date{} 

\begin{document}

\maketitle


\begin{abstract}
This work introduces a framework for quantifying the information content of logical propositions through the use of implication hypergraphs. We posit that a proposition's informativeness is primarily determined by its relationships with other propositions; specifically, the extent to which it implies or derives other propositions. To formalize this notion, we develop a framework based on implication hypergraphs, that seeks to capture these relationships. Within this framework, we define propositional information, derive some key properties, and illustrate the concept through examples. While the approach is broadly applicable, mathematical propositions emerge as an ideal domain for its application due to their inherently rich and interconnected structure. We provide several examples to illustrate this and subsequently discuss the limitations of the framework, along with suggestions for potential refinements.
\end{abstract}

\section{Introduction}

Information theory provides a mathematical definition of the 'information' of an event, expressing it as a function of the event's probability. This formalism was first introduced in Claude Shannon's seminal work, \emph{A mathematical theory of communication} \cite{https://doi.org/10.1002/j.1538-7305.1948.tb01338.x}. In this context, we extend this concept to logical propositions: how can we quantify the information content of a proposition? An information-theoretic perspective might suggest treating logical propositions as events and then calculating their Shannon information. However, a key challenge arises, as there is no straightforward probability distribution that can be assigned to these 'events' (logical propositions).

One potential solution involves heuristically estimating the probability of a proposition's 'occurrence', perhaps by analyzing its frequency within a relevant corpus or body of literature. While intriguing, this approach is both complex and, more critically, misaligned with our intuitive understanding of what makes a statement informative. In this framework, a proposition with a low probability (or low frequency) would inherently be assigned a higher information value. However, from an intuitive standpoint, it is unclear why rarity alone should determine the informativeness of a proposition.

Several studies have explored methods for representing or quantifying logical propositions in various forms. For example, \cite{logicasvector} introduces a framework where propositions are modeled as vectors in a logical space, while \cite{propspace} proposes the concept of a proposition space with associated logical transformations. In \cite{inproceedings}, the focus is on exploring 'degrees' of information and contradiction within a propositional framework, whereas \cite{propformulae} uses implication hypergraphs for the representation and simplification of propositional formulae. However, to the best of our knowledge, none of the works explicitly addresses the task of numerically quantifying the information content of a logical proposition.

We propose that the informativeness of a proposition is intrinsically linked to its relationships with other propositions. More precisely, a proposition's informativeness can be quantified by the number of other propositions derivable from it, that is, the extent of its implications. This perspective aligns with the conceptual notion of an informative statement: one that establishes numerous meaningful connections or implications.

To formalize this idea, we use implication hypergraphs, which provide a flexible and strctured framework to represent these relationships. Within this framework, we define propositional information, derive some key results and properties, and illustrate the concept through examples. Prior to this, we introduce some preliminary definitions and terminology that will be used throughout the text.

While the framework we present is general, its primary focus is on mathematical propositions. Mathematics presents a vast and intricate web of interconnected propositions, where each connection stems from a previously established theorem or result. This inherent structure makes it an ideal domain for using implication hypergraphs to compute propositional information. As a simple example, stating that a set is 'compact' is usually far more informative than stating it is 'closed'. This is because, firstly, compactness implies closure (in a Hausdorff space), and secondly, compactness plays a central role in numerous theorems and results in analysis. The tools developed in this work not only enable such comparisons of informativeness but also provide a method to quantify it. In Section~6, we present a couple of examples illustrating this idea.


\section{Preliminary Definitions}

\begin{definition}[Hypergraph]\label{def:hypergraph}
A hypergraph generalizes a graph by allowing a hyperedge to connect any number of nodes. It is represented as \( H = (V, E) \), where \( V \) is the set of vertices, and \( E \) is the set of hyperedges.

For a directed hypergraph, a hyperedge is an ordered pair \( e = (T, H) \) with \( T, H \subseteq V \) as the \emph{tail} (source nodes) and \emph{head} (target nodes), respectively. We denote \( t:E \to P(V) \) and \( h:E \to P(V) \) such that \( t(e) = T \) and \( h(e) = H \). Additionally, \( L(H) \) represents the set of leaf nodes of \( H \).
\end{definition}

As a simple example of a directed hypergraph, consider the hypergraph \(H = (V, E)\), where \(V = \{v_1, v_2, v_3\}\) is the set of vertices, and \(E = \{(\{v_1, v_2\}, \{v_3\})\}\) is the set of directed hyperedges. The directed hyperedge \((\{v_1, v_2\}, \{v_3\})\) consists of the \emph{tail} \(\{v_1, v_2\}\), which represents the source vertices, and the \emph{head} \(\{v_3\}\), which represents the target vertex.

\begin{definition}[Implication Hypergraph]\label{def:implication_hypergraph}
An \emph{implication hypergraph} is a specialized directed hypergraph used to represent implication relationships between logical propositions. Formally, it is a pair:
\[
H = (V, E),
\]
where:
\begin{itemize}
    \item \(V\) is the set of vertices, each representing a logical proposition,
    \item \(E\) is the set of directed hyperedges, where each hyperedge \(e = (T, H)\) represents the relation:
    \[
    \bigwedge_{P\in T} P \implies \bigwedge_{Q\in H}Q
    \]
\end{itemize}
Directed hyperedges in implication hypergraphs represent logical relationships, where the set of propositions in the tail collectively imply the proposition(s) in the head. Figure~\ref{fig:impli_hyp} shows a simple example of an implication hypergraph.

\end{definition}

\begin{figure}[h!] 
    \centering
    \includegraphics[width=0.37\textwidth]{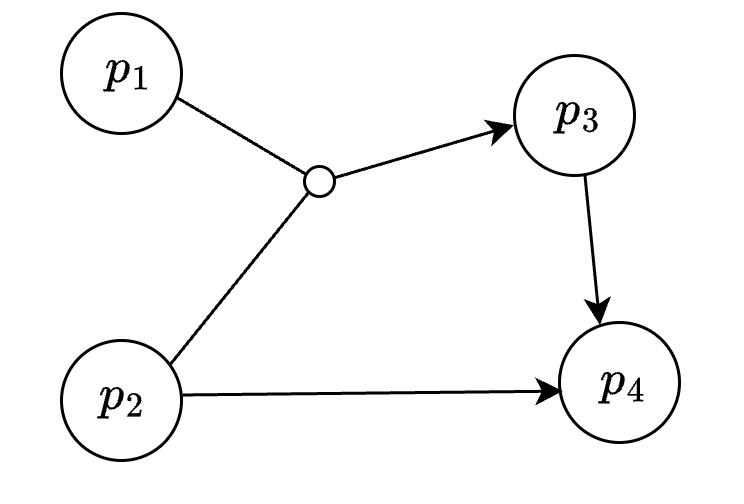}
    \caption{An example of an implication hypergraph with hyperedges \( (\{p_1, p_2\}, \{p_3\}) \), \( (\{p_3\}, \{p_4\}) \), and \( (\{p_4\}, \{p_2\}) \).}
    \label{fig:impli_hyp}
\end{figure}

\begin{definition}[Strict Implication Hypergraph]\label{def:implication_hypergraph}
A \emph{strict} implication hypergraph is an implication hypergraph \( H = (V, E) \) that satisfies the following property: 
\[
\forall v,w \in V, (v,w)\in E\implies(w,v)\notin E.
\]
\end{definition}

Implication hypergraphs have been previously employed to represent logical propositions (\cite{propformulae}, \cite{Verdee2024}). In this work, we adopt them as a tool to aid in the definition of propositional information. We start by introducing a specific type of implication hypergraph, which we call a \emph{minimal} implication hypergraph. It is characterized by the absence of redundant or trivial hyperedges.

\begin{definition}[Minimal Implication Hypergraph]\label{def:implication_hypergraph}
A \emph{minimal} implication hypergraph is an implication hypergraph \( H = (V, E) \) that satisfies the following properties:

\begin{enumerate}[label=(\roman*)]
    \item \(\forall v \in V, (v,v) \notin E \).
    \item If \( v \in V\) and \( u \) is a descendant of \( v \) but not an immediate neighbor, then $(v,u)\notin E.$
    \item If $e = (T,H)\in E$, then $e'=(T',H)\notin E$ where $T\subset T'$.
\end{enumerate}
Figure~\ref{fig:impli_hyp} illustrates a minimal implication hypergraph, while Figure~\ref{fig:non-strict} shows a non-minimal implication hypergraph.
\end{definition}

We denote the set of minimal strict implication hypergraphs by \( \mathcal{H} \). Henceforth, our focus will primarily be on implication hypergraphs belonging to this set.

\begin{figure}[h!] 
    \centering
    \includegraphics[width=0.38\textwidth]{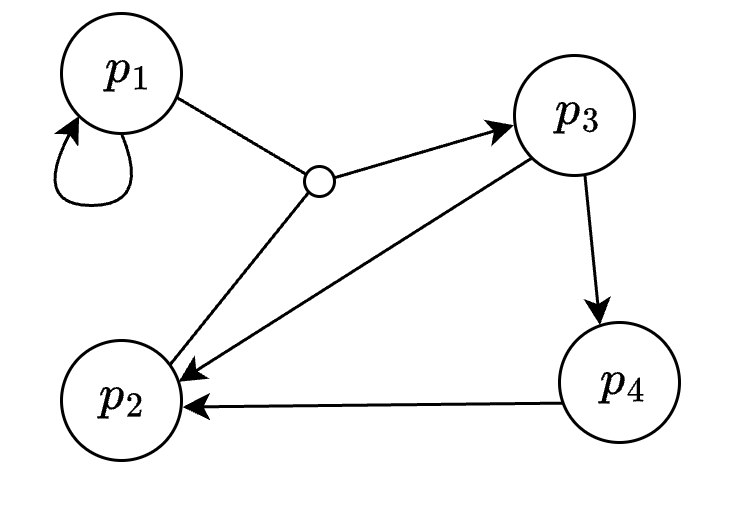}
    \caption{An example of a non-minimal implication hypergraph.}
    \label{fig:non-strict}
\end{figure}

\section{Propositional Information}

In our model of propositional information, we are guided by the following principles:

\begin{enumerate}[label=(\roman*)]
    \item If \( P \implies Q \), then \( P \) encompasses all the information contained in \( Q \).
    \item For a leaf node (proposition) \( Q \), the information it contains is a fixed, strictly positive value, denoted by \( \nu \).
    \item If \( P \implies Q \) but \( Q \not\implies P \), then \( P \) contains strictly more information than \( Q \). This additional information is in part quantified by a strictly positive value, denoted by \( \epsilon \).
    \item If \( P_1, P_2, \ldots, P_n \) collectively imply \( Q \), then each \( P_i \) contributes an equal share of the information leading to \( Q \), specifically \(\frac{1}{n}\) of the total information derived from \( Q \).
\end{enumerate}

The final point generalizes the first principle to scenarios involving hyperedges, where multiple propositions collectively imply a single proposition.

Before defining propositional information, we first define the adjacency matrix of an implication hypergraph. 

\begin{definition}[Adjacency Matrix]\label{def:prop_info}
Let \( H = (V, E) \) be an implication hypergraph. We denote its vertices by \( V = \{p_1,p_2,\ldots,p_n\} \). 
Let \( E_i = \{e \in E \mid p_i \in h(e)\} \). Then the adjacency matrix of \( H \), \( A \), is defined by:
\[
A_{ij} = \sum_{\substack{e \in E_i}} \frac{1}{|h(e)|} \mathbf{1}_{p_j \in t(e)}
\]
for \( 1 \leq i, j \leq n \).
\end{definition}

As an example, the adjacency matrix $A$ corresponding to the implication hypergraph shown in Figure~\ref{fig:impli_hyp} is:
\[
A = \begin{bmatrix}
0 & 0 & \nicefrac{1}{2} & 0 \\
0 & 0 & \nicefrac{1}{2} & 1 \\
0 & 0 & 0 & 1 \\
0 & 0 & 0 & 0
\end{bmatrix}
\]

\begin{definition}[Propositional Information]\label{def:prop_info}
Let $H\in \mathcal{H}$, and $\nu, \epsilon>0$. Let $A$ denote the adjacency matrix of $H=(V,E)$, where $V=\{p_1,\ldots,p_n\}$. For $p_i\in V$, we define its propositional information as follows:
\[
\mathcal{I}_{\nu, \epsilon}(p_i) =
\begin{cases}
\displaystyle \sum_{j=1}^{n} A_{ij} (\mathcal{I}_{\nu, \epsilon}(p_j) + \epsilon), & \text{if } p_i \notin L(H), \\
\nu, & \text{if } p_i \in L(H).
\end{cases}
\]

\end{definition}

The definition of propositional information is implicit by nature, requiring the resolution of a system of equations where \( \mathcal{I}_{\nu, \epsilon}(p_1), \ldots, \mathcal{I}_{\nu, \epsilon}(p_n) \) are the unknowns. Notably, this definition neither ensures the existence of a solution nor guarantees that the solutions, if they exist, are strictly positive. The proposition below addresses the issue of existence.

\begin{proposition}\label{prop:well-defined}
\normalfont
Let $H\in \mathcal{H}$, and $\nu, \epsilon>0$. Let $A$ denote the adjacency matrix of $H=(V,E)$, where $V=\{p_1,\ldots,p_n\}$. Let \( \mathbf{1} \) be the vector of $n$ ones, and \( l \) be a vector where the \( i \)-th entry is 1 if \( p_i \) is a leaf node, and 0 otherwise. If \( A - I \) is invertible, or equivalently if 1 is not an eigenvalue of \( A \), then \( \mathcal{I}_{\nu, \epsilon} = (\mathcal{I}_{\nu, \epsilon}(p_1), \ldots, \mathcal{I}_{\nu, \epsilon}(p_n)) \) is well-defined, and

\[
\mathcal{I}_{\nu, \epsilon} = (I - A)^{-1} \left( \epsilon A \mathbf{1} + \nu l \right).
\]
\end{proposition}

Considering once again the implication hypergraph in Figure~\ref{fig:impli_hyp}, we observe that \( A - I \) is invertible. Consequently, \( \mathcal{I}_{\nu, \epsilon} = (\mathcal{I}_{\nu, \epsilon}(p_1), \ldots, \mathcal{I}_{\nu, \epsilon}(p_4)) \) is well-defined. Here, \( l = (0, 0, 0, 1) \). Applying the formula, we obtain:
\[
\renewcommand{\arraystretch}{1.2}
\mathcal{I}_{\nu, \epsilon} = \begin{bmatrix}
\tfrac{1}{2}\nu + \epsilon\\
\tfrac{3}{2}\nu + 2\epsilon \\
\nu + \epsilon\\
\nu
\end{bmatrix}
\]

Since we associate the value of information to be positive, we would ideally aim to obtain all positive values through Definition~\ref{def:prop_info}. The following definition specifies those implication hypergraphs that satisfy this condition exactly.

\begin{definition}[Configured Implication Hypergraph]\label{def:rep}
Let \( H = (V, E) \in \mathcal{H} \) and let \( \nu, \epsilon > 0 \). The vertices of \( H \) are denoted by \( V = \{p_1, p_2, \ldots, p_n\} \). The hypergraph \( H \) is called a \emph{configured} implication hypergraph if \( \mathcal{I}_{\nu, \epsilon}(p_i) > 0 \) for all \( i \in \{1, \ldots, n\} \), where \( \mathcal{I}_{\nu, \epsilon}(p_i) \) is as defined in Definition~\ref{def:prop_info}.
\end{definition}

\begin{theorem}\label{thm:nec}
\normalfont
Let \( H \in \mathcal{H} \), and \( A \) denote its adjacency matrix. If \( H \) is \emph{configured}, then for all \( i \in \{1, \ldots, n\} \),
\[
0 \leq \sum_{j=1}^{n} A_{ij} A_{ji} < 1.
\]
\end{theorem}

\begin{proof}
Let $ H = (V, E)$, where \( V = \{p_1, p_2, \ldots, p_n\} \). Since all values of $A$ are positive by definition, the left inequality follows immediately. We will now consider the right inequality. Let $k \in \{1,\ldots,n\}$. If $p_k\in L(H)$, then for all $j \in \{1,\ldots,n\}$, $A_{jk}=0$, and hence the inequality is easily verified. Conversely, suppose that $p_k\notin L(H)$. Then, by Definition~\ref{def:prop_info}, we know that
\begin{align*}
\mathcal{I}_{\nu, \epsilon}(p_k) &= \sum_{j=1}^{n} A_{kj} (\mathcal{I}_{\nu, \epsilon}(p_j) + \epsilon) \\
&= \sum_{j=1}^{n} A_{kj} \left(\sum_{l=1}^{n} A_{jl}(\mathcal{I}_{\nu, \epsilon}(p_l) + \epsilon) + \epsilon\right) \\
&= \sum_{j=1}^{n} A_{kj}A_{jk} \mathcal{I}_{\nu, \epsilon}(p_k) \: + \:C,
\end{align*}

where $C>0$, since for all $i \in \{1,\ldots,n\}$, $\mathcal{I}_{\nu, \epsilon}(p_i) > 0$, as $H$ is \emph{configured}. Thus, we get
\begin{align*}
&\mathcal{I}_{\nu, \epsilon}(p_k)\left(1-\sum_{j=1}^{n} A_{kj}A_{jk}\right)=C.
\end{align*}
Finally, since $\mathcal{I}_{\nu, \epsilon}(p_k) > 0$,
\begin{align*}
\sum_{j=1}^{n} A_{kj}A_{jk} < 1.
\end{align*}

\end{proof}

\begin{remark}
The necessary condition in Theorem~\ref{thm:nec} is not sufficient. Consider the following example.

\begin{figure}[H] 
    \centering
    \includegraphics[width=0.36\textwidth]{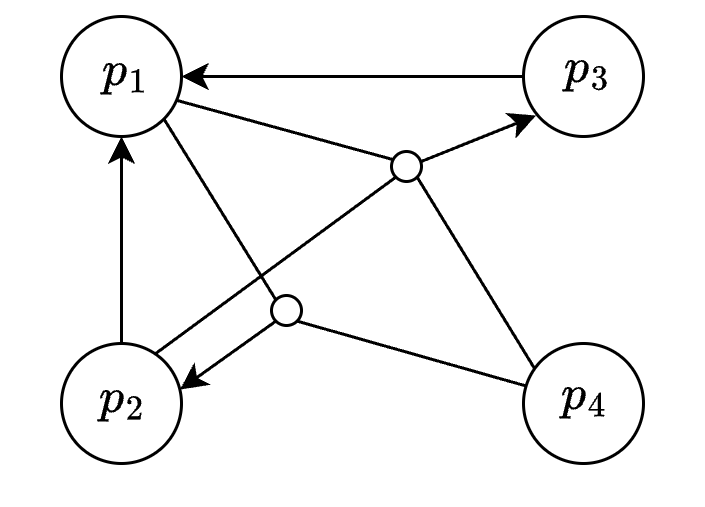}
    \label{fig:non-strict}
\end{figure}

Its adjacency matrix $A$ is given by:
\[
\renewcommand{\arraystretch}{1.2}
A = \begin{bmatrix}
0 & \nicefrac{1}{2} & \nicefrac{1}{3} & 0\\
1 & 0 & \nicefrac{1}{3} & 0\\
1 & 0 & 0 & 0\\
0 & \nicefrac{1}{2} & \nicefrac{1}{3} & 0
\end{bmatrix}
\]

and it satisfies the following condition:
\[ \forall i \in \{1, \ldots, n\},\;
0 \leq \sum_{j=1}^{n} A_{ij} A_{ji} < 1.
\]
Since $p_1$ is not a leaf node:
{\small
\begin{align*}
\mathcal{I}_{\nu, \epsilon}(p_1) &= \frac{1}{2}\mathcal{I}_{\nu, \epsilon}(p_2)+\frac{1}{2}\epsilon + \frac{1}{3}\mathcal{I}_{\nu, \epsilon}(p_3)+\frac{1}{3}\epsilon\\
&= \frac{1}{2}\mathcal{I}_{\nu, \epsilon}(p_1)+\frac{1}{6}\mathcal{I}_{\nu, \epsilon}(p_3)+\frac{1}{3}\mathcal{I}_{\nu, \epsilon}(p_1)+C\\
&=\frac{1}{2}\mathcal{I}_{\nu, \epsilon}(p_1)+\frac{1}{6}\mathcal{I}_{\nu, \epsilon}(p_1)+\frac{1}{3}\mathcal{I}_{\nu, \epsilon}(p_1)+C,
\end{align*}
}
where $C>0$. Thus we finally obtain $\mathcal{I}_{\nu, \epsilon}(p_1)=\mathcal{I}_{\nu, \epsilon}(p_1)+C$, which is impossible. Thus $\mathcal{I}_{\nu, \epsilon}$ is not well-defined, and hence the implication hypergraph is not representable.

\end{remark}

\begin{theorem}\label{thm:example2}
\normalfont
Let \( H=(V, E)\in \mathcal{H} \), where \( V = \{p_1, p_2, \ldots, p_n\} \). Let \( A \) denote its adjacency matrix. If for all \( k \in \{1, \ldots, n\} \), the diagonal of \( A^k \) is zero, i.e., \( \text{diag}(A^k) = \mathbf{0} \), then \( H \) is \emph{configured}.
\end{theorem}

\begin{proof}
The case \( n = 1 \) is straightforward, as the graph consists of a single node, which is necessarily a leaf node. Consequently, its information value is \( \nu > 0 \). Now, let us consider the case where \( n \geq 2 \).

We prove by contradiction that \( 1 \) cannot be an eigenvalue of \( A \). Assume, for the sake of contradiction, that \( 1 \) is an eigenvalue of \( A \). Since the eigenvalues of \( A^2 \) are the squares of the eigenvalues of \( A \), this would imply that \( \text{trace}(A^2) > 0 \). However, by our hypothesis, we have \( \text{diag}(A^2) = \mathbf{0} \), leading to a contradiction. Therefore, \( 1 \) cannot be an eigenvalue of \( A \). Consequently, by Proposition~\ref{prop:well-defined}, \( \mathcal{I}_{\nu, \epsilon} = (\mathcal{I}_{\nu, \epsilon}(p_1), \ldots, \mathcal{I}_{\nu, \epsilon}(p_n)) \) is well-defined.

Let $p_i \in V$. If $p_i \in L(H)$, then $I_{\nu, \epsilon}(p_i) = \nu > 0$. Suppose $p_i \notin L(H)$. To show that $I_{\nu, \epsilon}(p_i) > 0$, we proceed by contradiction. Suppose that $I_{\nu, \epsilon}(p_i) \leq 0$. By Definition~\ref{def:prop_info}:
\[
I_{\nu, \epsilon}(p_i)=\sum_{j=1}^{n} A_{ij} (\mathcal{I}_{\nu, \epsilon}(p_j) + \epsilon).
\]
For all \( i, j \in \{1, \ldots, n\} \), \( A_{ij} \geq 0 \). Let \( S \subset \{1, \ldots, n\} \) be the set of indices \( s \) such that \( A_{is} > 0 \) for all \( s \in S \). Since \( p_i \notin L(H) \), we have \( S \neq \emptyset \). Therefore, there exists \( i_1 \in S \) such that \( \mathcal{I}_{\nu,\epsilon}(p_{i_1}) \leq 0 \). Note that there exists a hyperedge \( e \in E \) such that \( i \in H(e) \) and \( i_1 \in T(e) \). By repeatedly applying this reasoning, we construct a sequence 
$Q = (p_{i_1}, p_{i_2}, \ldots)$ of vertices such that:
\begin{enumerate}
    \item For each pair of consecutive vertices \( p_{i_k}, p_{i_{k+1}} \) in \( Q \), there exists a hyperedge connecting them.
    \item \( \forall p_{i_k} \in Q \), \( \mathcal{I}_{\nu,\epsilon}(p_{i_k}) \leq 0 \).
    \item Q terminates at $p_{i_k}$ if and only if $p_{i_k}\in L(H)$.
\end{enumerate}

Since \( \operatorname{diag}(A^k) = \mathbf{0} \) for all \( k \in \{1, \ldots, n\} \), \( H \) is acyclic. Hence, the sequence \( Q \) must be finite. Let $ Q = (p_{i_1}, p_{i_2}, \ldots, p_{i_m}), $
where \( m < n \), and consequently \( p_{i_m} \in L(H) \). However, if \( p_{i_m} \in L(H) \), then 
$ \mathcal{I}_{\nu,\epsilon}(p_{i_m}) = \nu > 0, $
which contradicts our earlier assumption that \( \mathcal{I}_{\nu,\epsilon}(p_{i_k}) \leq 0 \) for all \( p_{i_k} \in Q \).

Thus, \( \mathcal{I}_{\nu,\epsilon}(p_i) > 0 \) holds for all \( p_i \notin L(H) \), completing the proof.

\end{proof}

\begin{remark}
The sufficient condition in Theorem~\ref{thm:example2} is not necessary. Consider the following example.
\begin{figure}[H] 
    \centering
    \includegraphics[width=0.3\textwidth]{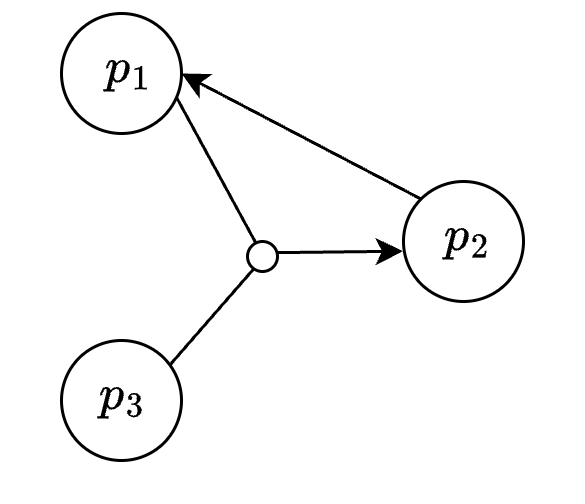}
    \label{fig:non-strict}
\end{figure}

Its adjacency matrix $A$ is given by:
\[
\renewcommand{\arraystretch}{1.2}
A = \begin{bmatrix}
0 & \nicefrac{1}{2} & 0\\
1 & 0 & 0\\
0 & \nicefrac{1}{2} & 0\\
\end{bmatrix}
\]
and \( \mathcal{I}_{\nu,\epsilon}\) can easily found to be equal to $(2\epsilon,3\epsilon,2\epsilon)$. However, $diag(A^2)=(\nicefrac{1}{2}, \nicefrac{1}{2},0)\neq\mathbf{0}$.

\end{remark}

\section{Applications}
The extensive corpus of mathematical theorems and results that underline the intricate interconnectedness of mathematical ideas serves as an ideal domain for applying the concepts developed in the preceding sections. Our framework enables the distillation of a set of mathematical results into a single implication hypergraph, allowing us to quantify the information content of each mathematical entity. However, obtaining the final numerical values requires specifying the parameters $\nu$ and $\epsilon$. 

\subsection{The role of $\epsilon$ and $\nu$}
In a strict implication hypergraph, if a proposition \( p_1 \) implies another proposition \( p_2 \), the information content of \( p_1 \) must be strictly greater than that of \( p_2 \), and \( \epsilon \) ensures this. However, if \( p_1 \) implies additional propositions beyond \( p_2 \), the cumulative implications naturally ensure that \( p_1 \)’s information content remains strictly greater than those it implies. This raises the question: why is \( \epsilon \) necessary?  

The key idea is that if the hypergraph is not 'complete', which it may never be, or if certain relationships are omitted, the computed information content may underestimate the 'true' value. The parameter \( \epsilon \) addresses this incompleteness by introducing a buffer, attempting to bring the calculated information content closer to the true value. Consequently, the choice of \( \epsilon \) reflects our belief in the completeness of the hypergraph: a relatively small \( \epsilon \) ((compared to $\nu$) indicates high confidence in its completeness, while a larger \( \epsilon \) compensates for greater uncertainty or perceived omissions.

\(\nu\), on the other hand, represents a unit of information content. All leaf nodes (propositions) are assigned the same information content, as they are, in a sense, the fundamental 'atoms' of information. While the exact value of \(\nu\) is not particularly significant, its relative magnitude compared to \(\epsilon\) is important, as discussed in the previous paragraph.

\subsection{Examples}
As an illustrative application to mathematical propositions, we first consider a set of 14 statements commonly encountered in an introductory course on analysis (e.g., "X is closed," "X is compact," etc.). Formally, we assume \( X \) is a finite-dimensional metric space, \( f \) is a function mapping \( X \) to another metric space \( Y \), and \( (u_n) \) is a sequence in \( X \). The mathematical propositions are constructed around these entities. The implication hypergraph corresponding to these 14 propositions is depicted below.

\begin{figure}[H] 
    \centering
    \includegraphics[width=1\textwidth]{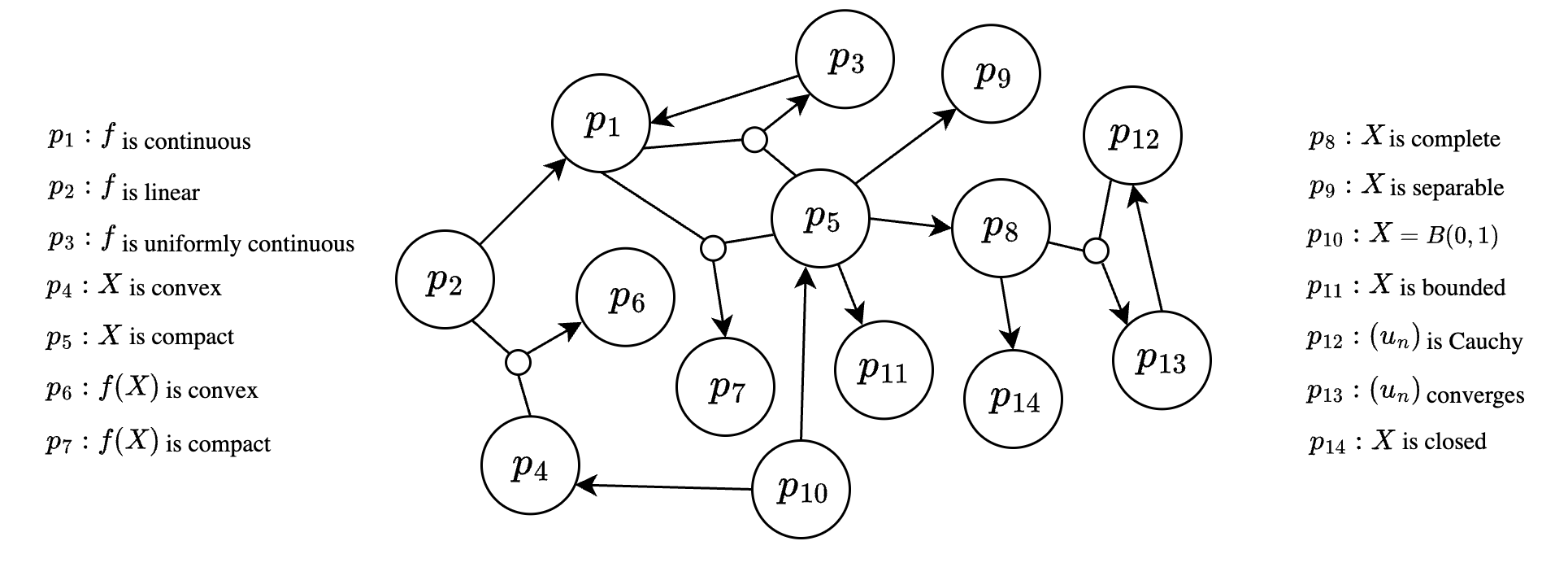}
    \label{fig:big-example}
\end{figure}

Given \( \epsilon, \nu > 0 \), the propositional information vector \( \mathcal{I}_{\nu, \epsilon} \) for this hypergraph is well-defined. For simplicity, we write \( \mathcal{I}_i = \mathcal{I}_{\nu, \epsilon}(p_i) \) for $i\in\{1,\ldots,n\}$. Using the formula provided in Proposition~\ref{prop:well-defined}, we obtain:

{\small
\[
\begin{array}{l@{\hspace{1cm}}l}
\mathcal{I}_1 = \nu + 3\epsilon & \mathcal{I}_8 = \nu + 3\epsilon \\
\mathcal{I}_2 = 1.5\nu + 4.5\epsilon & \mathcal{I}_9 = \nu \\
\mathcal{I}_3 = \nu + 4\epsilon & \mathcal{I}_{10} = 4.5\nu + 10.5\epsilon \\
\mathcal{I}_4 = 0.5\nu + 0.5\epsilon & \mathcal{I}_{11} = \nu \\
\mathcal{I}_5 = 4\nu + 8\epsilon & \mathcal{I}_{12} = 2\epsilon \\
\mathcal{I}_6 = \nu & \mathcal{I}_{13} = 3\epsilon \\
\mathcal{I}_7 = \nu & \mathcal{I}_{14} = \nu \\
\end{array}
\]
}

To illustrate the magnitudes, let us consider the example with \( \nu = 0.5 \) and \( \epsilon = 1 \). Substituting these values into the expressions derived earlier, we compute the information content of the 14 mathematical propositions:

{\small
\[
\begin{array}{l@{\hspace{1cm}}l}
f \text{ is continuous}: 3.5 & X \text{ is complete}: 3.5 \\
f \text{ is linear} : 5.25 & X \text{ is separable} : 0.5 \\
f \text{ is uniformly continuous} : 4.5 & X=\mathcal{B}(0,1) : 12.75 \\
X \text{ is convex} : 0.75 & X \text{ is bounded} :0.5 \\
X \text{ is compact} : 10 & (u_n) \text{ is Cauchy} : 2 \\
f(X) \text{ is convex} : 0.5 & (u_n) \text{ converges}:3\\
f(X) \text{ is compact} : 0.5 & X \text{ is closed} = 0.5 \\
\end{array}
\]
}

As another example, we consider a set of 12 statements commonly encountered in an introductory course on optimization. Formally, we assume $X\subset \mathbb{R}^n$, $f$ a function from $X$ to $\mathbb{R}$, and $x^*\in X$. We also consider the following optimization problem:
\[
\begin{array}{ll}
\text{minimize} & f(x) \\
\text{such that} & x \in X
\end{array}
\tag{P}
\]

\begin{figure}[H] 
    \centering
    \includegraphics[width=1\textwidth]{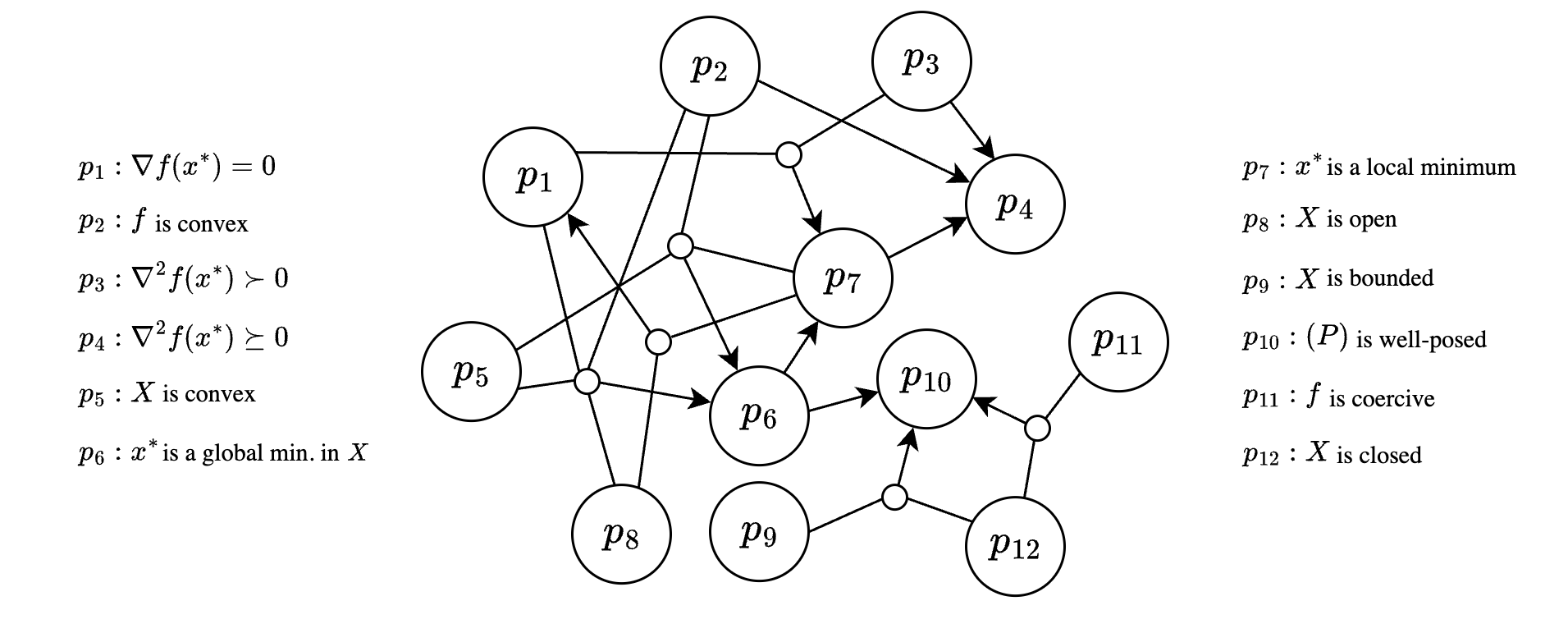}
    \label{fig:big-example2}
\end{figure}

Given \( \epsilon, \nu > 0 \), the propositional information vector \( \mathcal{I}_{\nu, \epsilon} \) for this hypergraph is well-defined. Like before, we write \( \mathcal{I}_i = \mathcal{I}_{\nu, \epsilon}(p_i) \) for $i\in\{1,\ldots,n\}$. Using the formula provided in Proposition~\ref{prop:well-defined}, we obtain:

{\small
\[
\begin{array}{l@{\hspace{1cm}}l}
\mathcal{I}_1 = 4\nu + 9.29\epsilon & \mathcal{I}_7 = 5\nu + 10.71\epsilon \\
\mathcal{I}_2 = 4.5\nu + 9\epsilon & \mathcal{I}_8 = 3.5\nu+8.57\epsilon \\
\mathcal{I}_3 = 3.5\nu + 6.86\epsilon & \mathcal{I}_{9} = 0.5\nu + 0.5\epsilon \\
\mathcal{I}_4 = \nu & \mathcal{I}_{10} = \nu \\
\mathcal{I}_5 = 3.5\nu + 8\epsilon & \mathcal{I}_{11} = 0.5\nu+0.5\epsilon \\
\mathcal{I}_6 = 6\nu+12.71\epsilon & \mathcal{I}_{12} = \nu+\epsilon \\
\end{array}
\]
}

To illustrate the magnitudes, let us consider the example with \( \nu = 1 \) and \( \epsilon = 0.5 \). Substituting these values into the expressions derived earlier, we compute the information content of the 12 mathematical propositions:

{\small
\[
\begin{array}{l@{\hspace{1cm}}l}
\nabla f(x^*)=0: 8.64 & x^* \text{ is a local minimum}: 10.36 \\
f \text{ is convex} : 9 & X \text{ is open} : 7.79 \\
\nabla^2f(x^*) \succ0 : 6.93 & X \text{ is bounded} : 0.75 \\
\nabla^2f(x^*) \succeq0 : 1 & (P) \text{ is well-posed} :1 \\
X \text{ is convex} : 7.5 & f \text{ is coercive} : 0.75 \\
x^* \text{ is a global min. in } X : 12.36 & X \text{ is closed}:1.5\\
\end{array}
\]
}

\section{Discussion}
It may perhaps be theoretically possible to condense most, if not all, of mathematics into an implication hypergraph similar to the ones presented in the previous section. Achieving this would require identifying a comprehensive set of mathematical propositions, simple yet versatile, that can serve as building blocks for representing the vast corpus of mathematical theorems and results. These propositions would serve as the nodes of this envisioned hypergraph. An intriguing reference for such an endeavor could be Bertrand Russell's \emph{Principia Mathematica} \cite{Russell1910-RUSPMV}, which systematically formalizes and constructs several fields of mathematics from the ground up. Such a hypergraph, if realized, could provide a striking visual representation of the interconnectedness of mathematical ideas. Furthermore, the concepts developed in this work could then quantify the information content of each of these fundamental mathematical propositions.

This work introduces a straightforward method for quantifying propositional information using implication hypergraphs, while leaving ample room for further exploration and refinement. Potential avenues for future research can include:
\begin{itemize}  
    \item Analysing the roles of \( \nu \) and \( \epsilon \) in depth, and establishing criteria for determining their optimal values. 
    \item Defining the information of the conjunction of two logical propositions, and eventually generalising it.  
    \item Revisiting the assumption that, when a group of propositions implies another proposition, each proposition in the group contributes equally to the implication. This oversimplification warrants further refinement.  
    \item Incorporating additional nuances into the current framework to more effectively capture the factors that contribute to a proposition's informativeness.
 
\end{itemize}

\section{Conclusion}
In this work, we introduced a method for quantifying the information content of a logical proposition using implication hypergraphs. We proposed that the information content of a proposition is closely tied to its implication relationships with other propositions; specifically, the more propositions it implies, the more informative it is. Building on this premise, we formally defined propositional information and derived several key properties and results. We then discussed how mathematical propositions provide an ideal domain for applying this framework, and subsequently illustrated this with two examples. Finally, we examined the limitations of the current framework and outlined potential avenues for refinement.

\bibliography{ref}

\end{document}